\documentclass{amsart} 
\usepackage{amssymb}
\usepackage{amsthm}
\usepackage{dsfont}
\usepackage{xcolor}
\usepackage[curve,matrix,arrow]{xy}
\theoremstyle{plain}\newtheorem{Theorem}{Theorem}[section]
\theoremstyle{plain}
\theoremstyle{plain}
\theoremstyle{plain}\newtheorem{Lemma}[Theorem]{Lemma}
\theoremstyle{plain}\newtheorem{Proposition}[Theorem]{Proposition}
\theoremstyle{definition}
\theoremstyle{definition}
\theoremstyle{definition}
\theoremstyle{definition}\newtheorem{Remark}[Theorem]{Remark}
\theoremstyle{definition}
\theoremstyle{plain}

    \def\OG{{\mathcal{O}G}}  \def\OGb{{\mathcal{O}Gb}}
    \def\OH{{\mathcal{O}H}}  
    \def\OP{{\mathcal{O}P}}

\def\CO{{\mathcal{O}}}

\def\C{{\mathbb C}}

\def\Q{{\mathbb Q}}
\def\Z{{\mathbb Z}}

\def\Aut{\mathrm{Aut}}                
\def\Br{\mathrm{Br}}              

\def\dim{\mathrm{dim}}

\def\End{\mathrm{End}}

\def\Ind{\mathrm{Ind}}

\def\Irr{\mathrm{Irr}}           

           \def\tenO{\otimes_{\mathcal{O}}}

\def\SL{\mathrm{SL}}

\def\Tr{\mathrm{Tr}}

\title{On character degrees of cyclic and Klein four defect blocks} 
\author{Markus Linckelmann} 
\date{\today}

\address{Markus Linckelmann \\
School of Science \& Technology \\
Department of Mathematics \\
City St George's, University of London \\
Northampton Square \\
London EC1V 0HB \\
United Kingdom}
\email{markus.linckelmann.1@city.ac.uk}



\begin{document}

\begin{abstract}
We show that a conjecture of Giannelli on character degrees of height zero
characters holds for blocks with a cyclic or Klein four defect group.
\end{abstract}

\maketitle

\section{Introduction} 

Let $p$ be a prime and $\CO$ a complete discrete valuation ring with 
an algebraically closed residue field $k$ of prime characteristic $p$ and
a field of fractions $K$ of characteristic $0$.  

\medskip
Throughout the paper, $G$ is a finite group, $B$ a block of $\OG$ with
a  defect group $P$, and $C$ is the block of $\CO N_G(P)$ which is the
Brauer correspondent of $B$. We set $b=1_B$ and $c=1_C$; that is, $b$, $c$
are the primitive idempotents in $Z(\OG)$, $Z(\CO N_G(P))$ satisfying
$B=\OGb$ and $C=\CO N_G(P)c$. We assume that $K$ is a splitting field
for $K\tenO B$ and $K\tenO C$ and related blocks. 
We denote by $\Irr(G)$ the set of characters of the simple $KG$-modules
and by $\Irr(B)$ the set of characters of the simple $K\tenO B$-modules.
We denote by $\Irr_0(B)$ the subset of height zero characters in $\Irr(B)$.

\medskip
The McKay conjecture for blocks predicts that there is a bijection
$\Irr_0(B)\cong$ $\Irr_0(C)$. Much effort has been put into the question
what additional properties such a bijection should have. 
Giannelli conjectured in \cite[Conjecture B]{Gia25} that  there is
a bijection $f :\Irr_0(B)\cong$ $\Irr_0(C)$ satisfying $f(\chi)(1)\leq \chi(1)$
for all $\chi\in$ $\Irr_0(B)$.
We show that this conjecture holds for $P$ cyclic and
for $p=2$ and $P$ a Klein four defect group (in both cases we have $\Irr_0(B)=$
$\Irr(B)$ and the McKay conjecture, which in these cases is equivalent
to Alperin's weight conjecture, is known to hold; see \cite{Da66}, \cite{Bra71}). 
We show slightly more
precisely, that the bijection $f$ can be chosen to be the character bijection
induced by a perfect isometry. We refer to \cite[Sections 9.2, 9.3]{LiBookII}
for background material on perfect isometries. 

\begin{Theorem} \label{thm1}
Suppose that the defect group $P$ of the block $B$ is cyclic, or that $p=2$ 
and $P$ is a Klein four group.  Then there is a perfect isometry
$\Phi : \Z\Irr(B)\cong$ $\Z\Irr(C)$ with the property that
$f(\chi)(1)\leq \chi(1)$ for any $\chi\in$ $\Irr(B)$,  where $f : \Irr(B)\cong$ $\Irr(C)$ 
is the bijection such that $\Phi(\chi) =$ $\delta(\chi)f(\chi)$ for some signs
$\delta(\chi)\in\{1,-1\}$, for all $\chi\in \Irr(B)$. In particular,
 \cite[Conjecture B]{Gia25} holds for blocks with a cyclic or Klein four defect group. 
\end{Theorem}

The basic strategy for proving Theorem \ref{thm1} is as follows: after
reviewing background material on source algebras and decomposition matrices
in Section \ref{background-Section}, we show in Section
\ref{source-Section}  that a proof of \cite[Conjecture B]{Gia25} follows 
from the analogous  inequalities between the dimensions of  simple modules, 
over $K$,   of the source algebras of $B$ and $C$. We  then 
show in Section \ref{frob-Section} that the  source algebra version of this 
conjecture holds  for all blocks satisfying Alperin's weight conjecture 
with an abelian defect group and a cyclic inertial
quotient $E$ which acts freely on $P\setminus \{1\}$, with either $p$ odd
or $|E|=$ $|P|-1$. We show more precisely that the map  $f$ can be chosen as
being induced by a perfect isometry.

 \section{Block theory background} 
 \label{background-Section}
 
We review the  standard facts which relate source algebras of
blocks and their Brauer correspondents.  We keep the notation and
hypotheses introduced at the beginning of the paper. We denote by $\Br_P$ the
Brauer homomorphism (see e.g.  \cite[Theorem 5.4.1]{LiBookI}). Given an
idempotent $i$ in the $P$-fixed point algebra $(\OG)^P$ with respect
to the conjugation action of $P$ on $\OG$, the condition $\Br_P(i)\neq 0$
is equivalent to requiring that $i\OG i$ has a direct summand isomorphic
to $\OP$ as an $\OP$-$\OP$-bimodule (cf. \cite[Lemma 5.8.8]{LiBookI}).

\begin{Proposition}[{\cite[3.5]{Puigpoint}, cf. \cite[Theorem 6.4.6]{LiBookII}}]
\label{iBi-Morita-Prop}
Let $i$ be an idempotent in $B^P$ such that $\Br_P(i)\neq 0$.
Then the $B$-$\OP$-bimodule $Bi$ and the $\OP$-$B$-bimodule $iB$
induce a Morita equivalence between $B$ and $iBi$.
\end{Proposition}

By a result of Picaronny and Puig in \cite{PiPu} 
(see e.g. \cite[Proposition 6.11.11]{LiBookII} for an expository account and 
further references),  the height of a character in a block can be read off the
source algebras of the block. For characters of height zero this remains true for 
source idempotents replaced by a slightly more general class of idempotents.

\begin{Proposition} \label{heightzero-Prop}
Let $i$ be an idempotent in $B^P$ such that $\Br_P(i)$ is a 
primitive idempotent in $kC_G(P)$. Let $X$ be a simple $K\tenO B$-module.
The character of $X$ has height zero if and only if $\dim_K(iX)$ is prime to $p$.
\end{Proposition}

\begin{proof}
By the assumptions on  $i$ and standard lifting theorems on idempotents,
we may write $i=i_0 + i_1$ with orthogonal idempotents
$i_0$, $i_1$ in $B^P$ such that $i_0$ is a source idempotent (satisfying
$\Br_P(i_0)=$ $\Br_P(i)$) and such that $\Br_P(i_1)=0$. 
Let $U$ be an $\CO$-free $B$-module such that $K\tenO U\cong$
$X$; that is, $U$ affords the character of $X$. 
Write $iU=$ $i_0U\oplus i_1U$; this is a direct sum of $\OP$-modules.
By \cite{PiPu} (or \cite[Proposition 6.11.11]{LiBookII}), 
the $K\tenO B$-module $X$ has a character of height zero if and only if $i_0U$ has
$\CO$-rank prime to $p$. Since $\Br_P(i_1)=0$, every indecomposable direct
summand of $i_1U$ has a vertex strictly contained in  $P$, hence an $\CO$-rank
divisible by $p$ by \cite[Theorem 5.12.13]{LiBookI} (which is a consequence of
Green's Indecomposability Theorem). 
Thus the $\CO$-rank of $i_0U$ is prime to $p$ if and only
if the $\CO$-rank of $iU$ is prime to $p$, hence if and only if $\dim_K(iX)$ is
prime to $p$. 
\end{proof}

Set $N=N_G(P)$.  Then $c$ is contained in $(\CO C_G(P))^N$.
Choose a block idempotent $e$ of $\CO C_G(P)$ such that $ec=e$, and
setting $H=$ $N_G(P,e)$, we have $c=\Tr^N_H(e)$  
(cf. \cite[Theorems 6.2.6, 6.7.6]{LiBookII}). Note that $e$ remains a block idempotent
of $\OH$ with $P$ as a normal defect group. Moreover, any primitive 
idempotent $j\in$ $\CO C_G(P)e$ is in fact a source
idempotent $j\in $ $(\OH e)^P$ as well as in $(\CO Nc)^P$, and we
have $j\OH j=$ $j\CO N j$ (cf. \cite[Theorem 6.8.3]{LiBookII}). 
Source algebras of blocks with a normal
defect group are fully understood:

\begin{Proposition}[{\cite[Theorem A]{Kue85}; cf. \cite[Theorem 6.14.1]{LiBookII}}] 
\label{normal-Prop}
 With the notation above,  we have an isomorphism of interior $P$-algebras
$$j\OH j \cong \CO_\alpha(P\rtimes E),$$
where $E\cong$ $N_G(P,e)/P\C_G(P)$ is an inertial quotient of $B$ lifted to
a subgroup of $\Aut(P)$ 
and where  $\alpha\in$ $H^2(E; k^\times)$, inflated to $P\rtimes E$ and
regarded as an element of $H^2(P\rtimes E; \CO^\times)$ via the canonical
group isomorphism $\CO^\times \cong$ $k^\times \times (1+J(\CO))$.
\end{Proposition}

The statement of  Proposition \ref{normal-Prop} uses  the fact 
that $E$ has order prime to $p$ (thanks to the assumption that $k$ is
algebraically closed), so the group $N_G(P,e)/PC_G(P)$, which a priori
is a subgroup of the outer automorphism group of $P$, lifts to an actual
automorphism group of $P$, uniquely up to conjugation in $\Aut(P)$. 
The canonical isomorphism $\CO^\times\cong$ $k^\times \times (1+J(\CO))$
exists again since $k$ is assumed algebraically closed, hence perfect. 
By a result of Fan and Puig in \cite{FanPu99}  there is a primitive idempotent
$f\in$ $B^H$ such that $\Br_P(fe)\neq 0$, and then $i=jf$ is a source
idempotent in $B^P$. Note that Proposition \ref{iBi-Morita-Prop}
applies in particular to the idempotents $bc$, $e$, and $f$  as defined above. 
As an immediate consequence, we have the following. 

\begin{Proposition}[{\cite[4.10]{FanPu99}, cf. \cite[Proposition 6.7.4]{LiBookII}}]
\label{jfi-Prop}

With the notation above, 
multiplication by $f$ induces a unital injective algebra homomorphism
$$\CO_\alpha(P\rtimes E) \cong jCj \to iBi$$
which is split as a homomorphism of $jCj$-$jCj$-bimodules. 
\end{Proposition}

Upon coefficient extension to $K$, the homomorphism in Proposition \ref{jfi-Prop}
becomes a unital injective algebra homomorphism between 
semisimple $K$-algebras, leading to some obvious comparison statements
about dimensions of simple modules. We note this for future reference:

\begin{Lemma} \label{dim-Lemma}
Let $A$, $A'$ be finite-dimensional split semisimple $K$-algebras, and let
$f : A\to A'$ be a unital injective algebra homomorphism.
\begin{itemize}
\item[{\rm (i)}]
For any simple $A$-module $X$ there exists a simple $A'$-module such 
that $\dim_K(X)\leq \dim_K(X')$.

\item[{\rm (ii)}]
For any simple  $A'$-module $X'$ there exists a simple $A$-module $X$
such that $\dim_K(X)\leq \dim_K(X')$.
\end{itemize}
\end{Lemma}

\begin{proof}
Identify $A$ to its image in $A'$ under $f$.
The hypotheses imply that for any simple $A$-module $X$ there exists a
simple $A'$-module $X'$ such that $X$ is isomorphic to a submodule of the 
restriction to $A$ of $X'$, whence (i). Similarly, for any simple $A'$-module
$X'$ there is a simple $A$-module $X$ which is isomorphic to a submodule
of $X'$ restricted to $A$, whence (ii).
\end{proof}

Recall that an irreducible character $\chi$ of $G$ is called {\em $p$-rational}
if its values are contained in $\Q(\zeta')$ for some root of unity $\zeta'$ of
order prime to $p$. It is well-known that $\chi$ is $p$-rational if and only
if all its generalised decomposition numbers are rational integers.
Since generalised decomposition numbers of an irreducible character of $B$
are invariants of the source algebra $iBi$, the property of being $p$-rational
can be read off the simple $K\tenO iBi$-module $iX$ corresponding to a
simple $K\tenO B$-module with character $\chi$.
Following Puig \cite{Puigpoint} (see e.g. \cite[Theorem 5.15.3]{LiBookI},
\cite[Remark 6.13.10]{LiBookII}) 
the generalised decomposition matrix of $B$ is equal to the square matrix
$$(\chi(u_\epsilon))_{\chi, u_\epsilon}$$
where $\chi$ runs over $\Irr(B)$, $u_\epsilon$ runs over a set of representatives
of the conjugacy classes of local pointed elements, and where $\chi(u_\epsilon)=$
$\chi(uj)$ for some (hence any) $j\in \epsilon$. Every local pointed group
on $B$ has a $G$-conjugate contained in $P_\gamma$, where $\gamma$ is the
local point of $P$ on $B$ containing the source idempotent $i$. Thus, in the
indexing of the generalised decomposition numbers,  we may choose
the $u_\epsilon$ in such a way that each $\epsilon$ contains an element
in $(iBi)^{\langle u\rangle}$, or equivalently, such that $u_\epsilon\in$ $P_\gamma$.

\begin{Lemma} \label{nilp-dec-Lemma}
Let $i\in B^P$ be a source idempotent. Let $u\in P$, and denote by $e_u$ the
unique block of $kC_G(u)$ such that $\Br_{\langle u\rangle}(i)e_u\neq 0$.
Suppose that $kC_G(u)e_u$ has a unique isomorphism class of simple modules. 
Then $\langle u\rangle$ has a unique local
point $\epsilon$ on $iBi$. Denote by $m$ the multiplicity of $\epsilon$ on $iBi$.
Let $j\in \epsilon$.  For every character of $B$ we have
$$\chi(ui) = m\cdot \chi(uj).$$
\end{Lemma}

\begin{proof}
By the assumptions,  $kC_G(u)e_u$ has a unique isomorphism class
of simple modules, hence a unique conjugacy class of primitive idempotents.
 Since $\Br_{\langle u\rangle}(j)$ is a primitive idempotent
in $kC_G(u)e_u$, it follows (e.g. from \cite[Theorem 4.7.1]{LiBookI})  that 
$\langle u\rangle$ has a unique local point on $iBi$. Let $J$ be a primitive 
decomposition of $i$ in $(iBi)^{\langle u\rangle}$. Let $j'\in J$. If $j'$ belongs to 
$\epsilon$, then $\chi(uj)=$ $\chi(uj')$. If $j$ does not belong to $\epsilon$, then 
$j$ belongs to a point of $\langle u\rangle$ which is not local. By 
\cite[2.2]{BrPu} (see also \cite[Theorem 5.12.16]{LiBookI})
we have $\chi(uj')=0$ in that case. Since $m$ is the number of elements in $J$
belonging to $\epsilon$, the result follows.
\end{proof}

\section{A source algebra version of Giannelli's conjecture} 
\label{source-Section}

We use the notation and hypotheses introduced at the beginning of the paper.
Slightly extending earlier notation, for $A$ an $\CO$-free $\CO$-algebra
of finite $\CO$-rank we denote by $\Irr(A)$ the set of isomorphism classes
of simple $K\tenO A$-modules, and by $\Irr_{p'}(A)$ the subset of $\Irr(A)$
consisting of the isomorphism classes of simple $K\tenO A$-modules of
dimension prime to $p$. With this notation, Proposition \ref{heightzero-Prop},
is equivalent to the following statement: given  an idempotent $i\in B^P$ such that
$\Br_P(i)$ is a primitive idempotent in $kC_G(P)$, 
the correspondence sending a simple $K\tenO B$-module $X$ to the
simple $K\tenO iBi$-module $iX$ induces a bijection
$$\Irr_0(B) \cong \Irr_{p'}(iBi).$$
In what follows, we
abusively use the same letters for isomorphism classes of simple modules
and representatives of these, whenever convenient.

\begin{Theorem} \label{thm2}
Let $j\in C^P$ and $i\in B^P$ be source idempotents of $C$ and $B$, respectively.
Suppose that there is a bijection
$g : \Irr_{p'}(iBi) \cong $ $\Irr_{p'}(jCj)$ satisfying $\dim_K(g(W))\leq$ $\dim_K(W)$
for any simple $iKGi$-module $W$ of dimension prime to $p$. 
Then the bijection $f :\Irr_0(B)\to $ $\Irr_0(C)$ induced by $g$ and the
standard Morita equivalence between $B$ and $iBi$ as well as between 
$C$ and $jCj$ satisfies $f(\chi)(1)\leq$ $\chi(1)$ for all $\chi\in$ $\Irr_0(B)$.
\end{Theorem}

We state the key argument for the proof of Theorem \ref{thm2}
separately.

\begin{Lemma} \label{thm2-Lemma}
With the notation of Theorem \ref{thm2}, let
$X$ be a simple $K\tenO B$-module and $Y$ a simple $K\tenO C$-module.
If $\dim_K(iY)\leq$ $\dim_K(iX)$, then
$\dim_K(Y)\leq$ $\dim_K(cX)\leq \dim_K(X)$.
\end{Lemma}

\begin{proof}
We note first the well-known fact that $B$ and $cBc$ are Morita equivalent  
(by Proposition \ref{iBi-Morita-Prop} applied to the idempotent $bc$), 
and $C$ is isomorphic to a direct summand of $cBc$ as a
$C$-$C$-bimodule (cf. \cite[Theorem 6.7.2]{LiBookII}).
In particular, multiplication by $b$ is an injective unital
algebra homomorphism $C\to cBc$ which is split as a $C$-$C$-bimodule
homomorphism. 

Any primitive idempotent in $\CO C_G(P)c$ is a source
idempotent in $C^P$ (cf. \cite[Theorem 6.14.1]{LiBookII}). Thus, if
$J$ is a primitive decomposition of $c$ in $\CO C_G(P)c$, then $J$ is
also a primitive decomposition of $c$ in $C^P$, and every idempotent in $J$ is
a source idempotent of $C$. It follows that
$$\dim_K(Y)=|J|\cdot \dim_K(jY).$$
The same argument yields
$$\dim_K(cX) = |J|\cdot \dim_K(jX).$$
Since multiplication by $b$ is an injective algebra homomorphism $C\to$ $cBc$,
it follows that the image $Jb$ of the set $J$ remains a
(not necessarily primitive) decomposition of $cb$ in $(cBc)^P$. 
Noting that $\Br_P(b)$ is the image of $c$ in $kN_G(P)$, every $j'\in J$ satisfies
$\Br_P(j'b)\neq 0$, so every $j'b$ is of the form $i'+i{''}$ for
some source idempotent $i'\in B^P$ and some idempotent
$i{''}$ in $B^P$ which is orthogonal to $i'$ and satisfies $\Br_P(i{''})=0$.
Thus $j'Bj'$ and the source algebra $iBi$ are Morita equivalent,
and any simple $K\tenO B$-module $X$ satisfies $\dim_K(iX)=\dim_K(i'X)\leq$
$\dim_K(j'X)=$ $\dim_K(jX)$.  Thus we have 
$$|J|\cdot\dim_K(iX) \leq |J|\cdot \dim_K(jX)= \dim_K(cX).$$
The result follows from combining the above (in-)equalities.
\end{proof} 

\begin{proof}[{Proof of Theorem \ref{thm2}}]
It follows from Proposition \ref{heightzero-Prop} that
the bijection $g$ induces a bijection $f'$ between $\Irr_0(cBc)$ and
$\Irr_0(C)$, where $\Irr_0(cBc)$ denotes abusively the set of
isomorphism classes of simple $K\tenO cBc$-modules of the form $cX$, where
$X$ is a simple $K\tenO B$-module with a character of height zero. 
By Lemma \ref{thm2-Lemma} the  bijection $f$ satisfies
$\dim_K(f'(Y))\leq \dim_K(Y)$ for any simple
$K\tenO C$-module with character in $\Irr_0(C)$. The result follows.
\end{proof}

\begin{Remark}
The condition on $g$ in Theorem \ref{thm2}  
seems in general genuinely stronger than the conclusion
for $f$ in that Theorem, so even if $B$ were to satisfy
\cite[Conjecture B]{Gia25}, it is not clear whether this yields a map $g$
as in Theorem \ref{thm2}. The issue is that in 
Lemma \ref{thm2-Lemma} we do not know whether the inequality
$\dim_K(Y)\leq\dim_K(X)$ implies $\dim_K(jY)\leq\dim_K(iX)$.
By Proposition \ref{heightzero-Prop} one could replace $i$ by
a slightly larger idempotent (for instance, one could replace $i$ by $jb$). 
\end{Remark}

Theorem \ref{thm2} applies to nilpotent blocks.

\begin{Theorem} \label{nilp-Theorem}
Assume that $B$ is nilpotent.  Let $i\in B^P$ and $j\in C^P$ be source idempotents.
Then there is a perfect isometry
$\Phi : \Z\Irr(C)\cong$ $\Irr(B)$ induced by a Morita equivalence between
$B$ and $C$ such that $\dim_K(jY)\leq$ $\dim_K(i\Phi(Y))$, and hence
$\dim_K(Y)\leq$ $\dim_K(\Phi(Y))$,  for all
simple $K\tenO C$-modules $Y$, where $\Phi(Y)$ denotes a simple 
$K\tenO B$-module with character $\Phi(\chi)$, with $\chi$ the character of $Y$.
In particular, \cite[Conjecture B]{Gia25} holds for nilpotent blocks. 
\end{Theorem}

\begin{proof}
This is an immediate  consequence of  Puig's structure theory of nilpotent blocks
(\cite{Punil}; cf. \cite[Section 8.11]{LiBookII}):
we have $jCj\cong$ $\OP$ and $jBj\cong$ $S\tenO \OP$, where $S=$ $\End_\CO(V)$
for some indecomposable endopermutation module $V$ with vertex $P$.
In particular, $iBi$ is isomorphic to a matrix algebra over $\OP$. This yields
the inequalities in the Theorem, and the last statement follows from
Theorem \ref{thm2}.
\end{proof}

Nilpotent blocks are a special case of {\em inertial blocks};  that is, blocks
which are Morita equivalent to their Brauer correspondent via a bimodule
with endopermutation source. If $B$ is inertial, then a source algebra of $B$
is of the form $S\tenO \CO_\alpha(P\rtimes E)$, where
$\CO_\alpha(P\rtimes E)$ is a source algebra of its Brauer correspondent
(with the notation from Proposition \ref{normal-Prop}). Thus, proceeding 
exactly as in the nilpotent block case, we obtain the following.

\begin{Theorem} \label{inertial-Theorem}
Assume that $B$ is inertial.  Let $i\in B^P$ and $j\in C^P$ be source idempotents.
Then there is a perfect isometry
$\Phi : \Z\Irr(C)\cong$ $\Irr(B)$ induced by a Morita equivalence between
$B$ and $C$ such that $\dim_K(jY)\leq$ $\dim_K(i\Phi(Y))$ for all
simple $K\tenO C$-modules $Y$, where $\Phi(Y)$ denotes a simple 
$K\tenO B$-module with character $\Phi(\chi)$, with $\chi$ the character of $Y$.
In particular, \cite[Conjecture B]{Gia25} holds for inertial blocks. 
\end{Theorem}

\section{Frobenius inertial quotient and abelian defect}
\label{frob-Section}

Theorem \ref{thm1} is an immediate consequence of the following result,
combined with Theorem \ref{thm2}.

\begin{Theorem} \label{thm3}
Suppose that the defect group $P$ of $B$ is nontrivial abelian and that
$B$ has a cyclic inertial quotient $E$ acting freely
on $P\setminus \{1\}$.  Suppose that either $p$ is odd or that $|E|=$ $|P|-1$.
Let $i$ be a source idempotent in $B^P$.
Suppose that $|\Irr(B)|=$ $|\Irr(P\rtimes E)|$. 
There is a perfect isometry $\Phi :\Z\Irr(B) \cong$ $\Z\Irr(P\rtimes E)$
such that  $f(\chi)(1)\leq \chi(i)$ for any $\chi\in$ $\Irr(B)$.
Here $f:\Irr(B)\cong$ $\Irr(P\rtimes E)$ is the bijection such
that $\Phi(\chi)=$ $\delta(\chi) f(\chi)$ for some signs $\delta(\chi)\in$
$\{1,-1\}$, for all $\chi\in$ $\Irr(B)$. 
In particular, \cite[Conjecture B]{Gia25} holds for $B$. 
\end{Theorem}

\begin{proof}
We note first that $\CO(P\rtimes E)$ is a source algebra of the Brauer 
correspondent $C$ of $B$; indeed, since $E$ is cyclic, the class of
$\alpha$ in Proposition \ref{normal-Prop}  is trivial.  Thus, as recalled in 
Proposition \ref{jfi-Prop},  we may identify $\CO(P\rtimes E)$ with its 
image in the source  algebra $iBi$.  We note further that $\chi(i)$ is the
dimension of the simple $iBi$-module $iX$, where $X$ is a simple
$K\tenO B$-module having $\chi$ as character. Thus, by Theorem \ref{thm2}, 
the existence of a bijection $f$ as stated in the Theorem  
implies \cite[Conjecture B]{Gia25}.

We need some background material from \cite[Section 10.5]{LiBookII}
in order to describe perfect isometries. In keeping with the notation in that
reference, we consider perfect isometries from $\Irr(P\rtimes E)$ to
$\Irr(B)$ and use these to construct an inverse to the map called $f$ in the
statement.

By  \cite[Theorem 10.5.10]{LiBookII} all characters in $\Irr(B)$ have
height zero (this follows, of course, also from the general proof of one
direction of Brauer's height zero conjecture in \cite{KeMa}).  
We have a partition $\Irr(P\rtimes E) = M \cup \Lambda$, defined as follows. 
The set  $M$ consists of the irreducible character with $P$ in their kernel.
The characters in $M$ lift the simple $k(P\rtimes E)$-modules,
and the characters in $\Lambda$ are of the form $\Ind^{P\rtimes E}_P(\zeta)$,
with $\zeta$ running over a set of representatives of the $E$-orbits in
the set of nontrivial irreducible characters of $P$. In particular, we have
$|\Lambda|=$ $\frac{|P|-1}{|E|}$. The characters in $M$ have degree $1$, and
the characters in $\Lambda$ have degree $|E|$. 

By \cite[Lemma 10.5.8]{LiBookII}, \cite[Remark 10.5.9]{LiBookII},
and \cite[Theorem 10.5.10]{LiBookII}
there is a subset $\{\chi_\lambda\ |\ \lambda\in\Lambda\}$ of $\Irr(B)$
with the property that if we choose any labelling $\{\chi_\mu\ |\ \mu\in M\}$
of a complement of this set in $\Irr(B)$,  then
there is a perfect isometry $\Phi : \Z\Irr(P\rtimes E)\cong$ $\Irr(B)$
such that
$$\Phi(\lambda) = \delta\chi_\lambda$$
for all $\lambda\in$ $\Lambda$ and some sign $\delta$, and
$$\Phi(\mu) = \delta(\mu)\chi_\mu$$
for all $\mu\in $ $M$ and some signs $\delta(\mu)$. 
We denote by $g : \Irr(P\rtimes E)\cong$ $\Irr(B)$ the bijection
induced by $\Phi$; that is, $g(\mu)=$ $\chi_\mu$ for $\mu\in M$, and
$g(\lambda) = \chi_\lambda$ for $\lambda\in$ $\Lambda$. 

Assume first that  $|\Lambda|=1$; that is, $|E|=|P|-1$. 
Denote by $\lambda$ the unique element 
in $\Lambda$. By  \cite[Lemma 10.5.7]{LiBookII}
we may choose $\chi_\lambda$ arbitrarily in $\Irr(B)$; thus, by
Lemma \ref{dim-Lemma} applied to $\CO(P\rtimes E)$ and its
image in $iBi$, we may choose $\chi_\lambda$ such that
$\lambda(1) \leq \chi_\lambda(i)$. Since all remaining characters
of $\Irr(P\rtimes E)$ are in the set $M$ of characters of degree $1$,
it follows that the inverse of the bijection $g$  satisfies the conclusion
in the statement.

Assume that $|\Lambda|\geq 2$. Then $|E|<|P|-1$, so by the assumptions on
$|E|$, we may assume that $p$ is odd.
Then the set $\{\chi_\lambda\ |\ \lambda\in\Lambda\}$
is uniquely determined (cf. \cite[Lemma 10.5.7]{LiBookII}). 
Since $p$ is odd, the set $\Lambda$ is
exactly the set of characters in $\Irr(P\rtimes E)$ which are not $p$-rational,
and $\{\chi_\lambda\ |\ \lambda\in\Lambda\}$ is the set of characters in $\Irr(B)$
which are not $p$-rational. 
Let $\lambda\in$ $\Lambda$. Since $\chi_\lambda$ is not $p$-rational, there
is a non-trivial element $u\in P$ and a primitive idempotent $j$ in
$(iBi)^{\langle u\rangle}$ belonging to a local point of $\langle u\rangle$
on $iBi$ such that $\chi(uj)\not\in$ $\Z$. Since $E$ acts freely on $P\setminus \{1\}$,
it follows that the unique block $e_u$ of $kC_G(u)$ satsifying 
$\Br_{\langle u\rangle}(i)e_u\neq 0$ is nilpotent (cf. \cite[Lemma 10.5.2]{LiBookII}),
hence has a unique isomorphism class of simple  modules. 
It follows from Lemma \ref{nilp-dec-Lemma} that $\chi_\lambda(ui)\not\in \Z$.
Thus if $X$ is a simple $K\tenO B$-module with character $\chi_\lambda$, then
the restriction of $iX$ to $K(P\rtimes E)$ has a character which is not $p$-rational,
so must have a constituent equal to $\lambda'$ for some $\lambda'\in$ $\Lambda$.
In particular, $\lambda'1(1)\leq$ $\chi_\lambda(i)$, where we note that $\chi_\lambda(i)=$
$\dim_K(iX)$.  Since the $\lambda\in$ $\Lambda$ all have the same degree, it
follows that $\lambda(1)\leq$ $\chi_\lambda(i)$ for all $\lambda \in \Lambda$.
Since $\mu(1)=1\leq \chi_\mu(i)$ for all $\mu\in $ $M$, it follows again
that the inverse of the bijection $g$ satisfies the conclusion of the statement.
\end{proof}

\begin{proof}[{Proof of Theorem \ref{thm1}}]
Suppose first that $P$ is cyclic. If $P=1$ there is nothing to prove, so
assume that $P$ is nontrivial.  The equality $|\Irr(B)|=$ $|\Irr(P\rtimes E)|$ 
goes  back to Dade \cite{Da66} (see e.g. \cite[Theorem 11.1.3]{LiBookII}). 
If $p=2$, then $E=1$; that is, the block
$B$ is nilpotent. In that case, Theorem \ref{thm1} follows from 
Theorem \ref{nilp-Theorem}. If $p>2$, then Theorem \ref{thm1} is
a special case of Theorem \ref{thm3}. Suppose next that $p=2$ and that
$P$ is a Klein four group. In that case the equality $|\Irr(B)|=$ $|\Irr(P\rtimes E)|$ 
is due to Brauer \cite{Bra71} (see e.g. \cite[Corollary 12.1.5]{LiBookII}). 
The inertial quotient $E$ is either trivial, or has order $3$.
In the first case $B$ is nilpotent, so this case of Theorem \ref{thm1}
follows as before from Theorem \ref{nilp-Theorem}. In the second case
we have $|E|=3=|P|-1$, so Theorem \ref{thm1} follows again from 
Theorem \ref{thm3}.
\end{proof}

\begin{Remark}
For $p=2$ one other case of interest covered by Theorem \ref{thm3} is
where $P$ is an elementary abelian $2$-group of rank $n$ and $E$ cyclic, 
generated by a Singer cycle (i.e. an automorphism of order $|P|-1$ of $P$). 
Then, by the results of McKernon in \cite{McKer}, 
$B$ is Morita equivalent via a bimodule with endopermutation source
to either $\CO(P\rtimes E)$ (that is, $B$ is inertial), or to the principal 
block of $\CO\SL_2(2^n)$.
\end{Remark}

\begin{Remark}
If the block $B$ of $\OG$ has an abelian defect group $P$, then Brou\'e's
Abelian Defect Conjecture predicts the existence of a derived
equivalence, hence a perfect isometry, between $B$ and its Brauer
correspondent $C$.
One obvious question is whether such a derived equivalence can
be chosen in such a way that the character bijection obtained from
the induced perfect isometry satisfies the inequalities in Theorem
\ref{thm1}, or possibly even those in Theorem \ref{thm3}.
\end{Remark}

\bigskip\noindent
{\bf Acknowledgements.} 
The author acknowledges support from  EPSRC grant EP/X035328/1. 
The author would like to thank Eugenio Giannelli
for some very helpful comments on this paper. The author would further
like to thank  the Department of
Mathematics of the University of Essex for organising  the conference 
AGGITatE 2025 that lead to writing this paper. 

\bigskip

\end{document}